\documentclass[11pt,twoside,a4paper]{article}
\topskip  \usepackage{amssymb}
\usepackage{amsmath}
\usepackage{amsthm}
\usepackage{hyperref}
\usepackage{rotating}
\usepackage{float}

\newtheorem{definition}{Definition}
\newtheorem{theorem}{Theorem}
\newtheorem{lemma}{Lemma}

\newtheorem{remark}{Remark}
\numberwithin{equation}{section} %To enable multi-level eguation numbering
\numberwithin{definition}{section}
\numberwithin{theorem}{section}
\numberwithin{subsection}{section}
\numberwithin{lemma}{section}
\numberwithin{corollary}{section}
\numberwithin{remark}{section}

\author{B. S. El-Desouky, R.S. Gomaa }

\title{New Multivariate Discrete Distributions \\
             UGAT Distributions and Their Applications in Reliability}
\begin{document}
\maketitle
\begin{abstract} In this paper, we introduce a new multivariate discrete distribution which called multivariate unification of generalied Apostol type distribution (UGAT). Several properties are studied as, moments, probability generating function and other properties. Also, reliability study of distribution are introduced. Maximum likelihood method is used to estimate parameters and numerical method is used to obtain (MLEs).\\


\end{abstract}

\textbf{\section{ Introduction}}
 Recently, EL-Desouky et al. \cite{El-Desouky} introduced unification of generalized Apostol Euler, Bernoulli and Genocchi polynomials and obtained its explicit formula as 
 \begin{equation}\label{eq1}
M^{(r)}_{n}(\beta;k;\overline{\alpha})=\frac{(-1)^{r}2^{r(1-k)}n!}{(n-rk)!}\sum\limits_{\ell_{1},\ell_{2},...,\ell_{r}=0}^{\infty}\frac{(\alpha_{0})^{\ell_{1}}(\alpha_{1})^{\ell_{2}}...(\alpha_{r-1})^{\ell_{r}}}{(\ell_{1}+\ell_{2}+...+\ell_{r}+\beta)^{rk-n}}.
\end{equation}
\\
In the present paper we state a new multivariate discrete distribution \textbf{(UGAT)} distribution using the explicit formula of the unification of generalized Apostol Euler, Bernoulli and Genocchi polynomials. Several properties of the new distribution \textbf{(UGAT)} have been established as the joint cumulative distribution function, moments, marginal probability function and other properties in section 2. Also, in section 2, we present sub models from the new distribution \textbf{(UGAT)}
\\
 In 1975, Nakagawa and Osaki \cite{Osaki} were the first to study a discrete life time distribution which is defined as the discrete counterpart of the usual continuous Weibull distribution. In 1982, Salvia and Bollinger \cite{Salvia} introduced basic results about discrete reliability and illustrated them with the simple discrete life distributions. The characterization of discrete distributions has been studied by  Roy, Gupta, Gupta in 1999 \cite{Gupta4}.  In 1997, Gupta, Gupta and Tripahi \cite{Gupta3} introduced wide classes of discrete distributions with increasing failure rate. In 1997, Nair and Asha \cite{Asha} introduced some classes of multivariate life distributions in discrete time.\\
\\
So, In section 3, we introduce some applications of \textbf{(UGAT)} distribution in reliability theory. Finally, Section 4 contains the parameter estimation using MLE. In section 5, we present a numerical result are obtained using real data. 

\textbf{\section{The model and statistical inference}}
In this section, we introduce the multivariate generalized Apostol-Euler Bernoulli and Genocchi distribution \textbf{(UGAT)} and we study its structural properties and statistical inference.
\begin{definition} The discrete random variables $X_{1},X_{2},...,X_{r}$ are said to have UGAT distribution with the joint probability mass function
\begin{multline}\label{eq2}
p(X_{1}=x_{1},X_{2}=x_{2},...,X_{r}=x_{r})=\frac{(-1)^{r}2^{r(1-k)}n!(\alpha_{0})^{x_{1}}(\alpha_{1})^{x_{2}}...(\alpha_{r-1})^{x_{r}}}{(n-rk)!(x_{1}+x_{2}+...+x_{r}+\beta)^{rk-n}M^{(r)}_{n}(\beta;k;\overline{\alpha}_{r})} \\
where \;\; x_{1},x_{2},...,x_{r}=0,1,..., \alpha_{0},\alpha_{1},...\alpha_{r-1}\in \mathbf{R^{+}},\beta >0, n \in N_{0}  =N \cup{\lbrace 0 \rbrace} \;\; and \;\; rk>n.  
\end{multline}
\end{definition}
We show that many well-known discrete distributions are special cases of \textbf{UGAT} family. Some of these are given below
\textbf{\subsection{Special models}}
\subsubsection{Lerch distribution}
 
$P(X_{1}=x)=$ {\Large $\frac{p^{x}}{(x+a)^{c}M_{n}^{(1)}(a+1;n+c;p)}$},  $x \in N$ 

where $ M_{n}^{(1)}(a+1;n+c;p)=${\Large $\sum\limits_{\ell=1}^{\infty}\frac{p^{\ell}}{(\ell+a)^{c}}
$}, see \cite{Zorn}.\\
\subsubsection{Hurwitz Lerch zeta distribution}
 $P(X_{1}=x)=$ {\Large $\frac{\theta ^{x}}{(x+a)^{s+1}M_{n}^{(1)}(a+1;n+s+1;\theta)}$}, \quad\quad $(x \in N$, $s \geq 0$, $0 \leq a \leq1$, $0< \theta \leq1$), \\
 \\
where $M_{n}^{(1)}(a+1;n+s+1;\theta)=$ {\Large$\sum\limits_{\ell=1}^{\infty}\frac{\theta^{\ell}}{(\ell+a)^{s+1}}$}, see \cite{Gupta2}.
\subsubsection{Good distribution}
$P(X_{1}=x)=$ {\Large$\frac{\theta ^{x}}{(x)^{s+1}M_{n}^{(1)}(1;n+s+1; \theta)}$}, \quad\quad  ( $x \in N$, $0<\theta<1$ ), 
\\
where $M_{n}^{(1)}(1;n+s+1;\theta)=$ {\Large$\sum\limits_{\ell=1}^{\infty}\frac{\theta^{\ell}}{(\ell)^{s+1}}$}, see \cite{Gupta2}.
\subsubsection{Hurwitz Zeta distribution}
$P(X_{1}=x)=${ \Large $\frac{1}{(x+b)^{\sigma}M_{n}^{(1)}(b;n+\sigma;1)}$},   $k \in N$\\
where $M_{n}^{(1)}(b;n+\sigma;1)=$ {\Large $\sum\limits_{\ell=0}^{\infty}\frac{1}{(\ell+b)^{\sigma}}$}, see \cite{Hu}.
\subsubsection{Zipf-Mandelbrot distribution}
 $P(X_{1}=x)=${\Large $\frac{1}{(x+a)^{c}M_{n}^{(1)}(a+1;n+c;1)},$}       ($x \in N$, $a>0$, $c>0$) \\
 where $M_{n}^{(1)}(a+1;n+c;1)=$ {\Large $\sum\limits_{\ell=1}^{\infty}\frac{1}{(\ell+a)^{c}}$}, see \cite{Zorn}

\subsubsection{discrete Pareto distribution}
 $P(X_{1}=x)=${\Large $\frac{1}{(x)^{c}M_{n}^{(1)}(1;n+c;1)},$} \quad\quad      $x \in N$ \\
 where $M_{n}^{(1)}(1;n+c;1)=$ {\Large $\sum\limits_{\ell=1}^{\infty}\frac{1}{(\ell+1)^{c}}$}, see \cite{Zorn}
\subsubsection{Geometric distribution}
$P(X_{1}=x)=$ \large {$\frac{p^{x-1}}{M_{n}^{(1)}(1;n;p)}$}$=p^{x-1}(1-p), $   \quad\quad   $x \in N$, ( $0<p<1$) \\  
where $M_{n}^{(1)}(1;n;p)=\sum\limits_{\ell=0}^{\infty}p^{m}$.\\
 \begin{lemma} Let $X_{i} \sim UGAT (\alpha_{i},\beta,k), i=1,2,...,r-1$. Then the marginal joint cumulative distribution function 
\begin{equation}\label{eq3}
P(X_{i}\leq x_{i})=1-(\alpha_{i-1})^{x_{i}+1}\frac{M^{(r)}_{n}(\beta+x_{i}+1;k;\overline{\alpha}_{r})}{M^{(r)}_{n}(\beta;k;\overline{\alpha}_{r})}.
\end{equation}
\end{lemma}
\begin{proof} Since\\
$
P(X_{i}\leq x_{i})=1-P(X_{i}> x_{i}) 
$\\
 \begin{equation*}
=1-\frac{(-1)^{r}2^{r(1-k)}n!}{(n-rk)!M^{(r)}_{n}(\beta;k;\overline{\alpha}_{r})}\sum\limits_{\ell_{i}=x_{i}+1}^{\infty}\sum\limits_{\substack{\ell_{1},...,\ell_{i-1} \\
\ell_{i+1},...\ell_{r}=0}}^{\infty}\frac{(\alpha_{0})^{\ell_{1}}(\alpha_{1})^{\ell_{2}}...(\alpha_{i-1})^{\ell_{i}} ...(\alpha_{r-1})^{\ell_{r}}}{(\ell_{1}+\ell_{2}+...+\ell_{i}+...+\ell_{r}+\beta)^{rk-n}}.
\end{equation*}
Put $\ell_{i}-x_{i}-1=\ell$, then \\
{ \small \begin{equation*}
 P(X_{i}\leq x_{i})=1-\frac{(-1)^{r}2^{r(1-k)}n!}{(n-rk)!M^{(r)}_{n}(\beta;k;\overline{\alpha}_{r})}\sum\limits_{\substack{\ell,\ell_{1},...,\ell{i-1}\\
 ,\ell_{i+1},...\ell_{r}=0}}^{\infty}\frac{(\alpha_{0})^{\ell_{1}}(\alpha_{1})^{\ell_{2}}....(\alpha_{i-1})^{x_{i}+\ell +1}...(\alpha_{r-1})^{\ell_{r}}}{(\ell_{1}+\ell_{2}+...+\ell+...+\ell_{r}+\beta+x_{i}+1)^{rk-n}}.
\end{equation*}}
From \eqref{eq2}, we obtain \eqref{eq3}.
\end{proof}
\begin{theorem} Suppose $X_{1},X_{2},...,X_{r}$ are mutually independent where $X_{i} \sim $\\
$ UGAT (\alpha_{i},\beta,k), i=1,2,...,r-1$. Then the multivariate cumulative distribution function is given by
{ \begin{equation}\label{eq4}
P(X_{1}\leq x_{1},X_{2} \leq x_{2},...,X_{r} \leq x_{r} )=\prod\limits_{i=1}^{r}\left(1-\frac{(\alpha_{i-1})^{x_{i}+1}M^{(r)}_{n}(\beta+x_{i};k;\overline{\alpha}_{r})}{M^{(r)}_{n}(\beta;k;\overline{\alpha}_{r})}\right).
\end{equation}}
\end{theorem}
\begin{proof} Since $X_{1},X_{2},...,X_{r}$ are mutually independent, then
\begin{eqnarray*}
P(X_{1} \leq x_{1},X_{2} \leq x_{2},...,X_{r} \leq x_{r})&=&\prod\limits_{i=1}^{r}P(X_{i} \leq x_{i})\\
&=&\prod\limits_{i=1}^{r}(1-P(X_{i}> x_{i})).
\end{eqnarray*}
From \eqref{eq3}, we obtain \eqref{eq4}.
 \end{proof}
 \begin{theorem} The marginal probability mass function of $X_{i}, i=1,2,...,r$ is given by
 \begin{equation}\label{eq5}
 P(X_{i}=x_{i})=(\alpha_{i-1})^{x_{i}}\left(\frac{\alpha_{i}M^{(r)}_{n}(\beta+x_{i}+1;k;\overline{\alpha}_{r})-M^{(r)}_{n}(\beta+x_{i};k;\overline{\alpha}_{r})}{M^{(r)}_{n}(\beta;k;\overline{\alpha}_{r})}\right).
 \end{equation}
 \end{theorem}
 \begin{proof} Using
 $$
  P(X_{i}=x_{i})= P(X_{i}\geq x_{i})- P(X_{i}>x_{i})
  $$\\
  and from \eqref{eq3}, we obtain \eqref{eq5}.
  \end{proof}
  \textbf{\subsection{Moments}}
  The moment generating function is defined by
  \begin{equation}\label{eq103}
  M_{\mathbf{X}}(\mathbf{t})=\frac{M_{n}^{(r)}(\beta;k;e^{\mathbf{t}}\overline{\alpha_{r}})}{M_{n}^{(r)}(\beta;k;\overline{\alpha_{r}})},
  \end{equation}
  where $\mathbf{t}=(t_{1},t_{2},...,t_{r})$ and $\mathbf{X}=(X_{1},X_{2},...,X_{r})$.\\
  The $\ell$-th moment can be obtained by partial differential of $\ell$ times the m.g.f in \eqref{eq103} with respect to $t_{i}$ and $t_{i}=0$ as following
  {\small \begin{eqnarray}\label{eq104}
  \mu^{'}_{\ell}&=&\frac{\partial^{\ell}}{\partial t_{i}^{\ell}}M_{\mathbf{X}}(\mathbf{t})|_{t_{i}=0}=E(X_{i}^{\ell})\\
  &=&\frac{(-1)^{r}2^{r(1-k)n!}}{(n-rk)!M_{n}^{(r)}(\beta;k;\overline{\alpha_{r}})}\sum\limits_{x_{1},x_{2},...,x_{i},...,x_{r}=0}^{\infty}\frac{x_{i}^{\ell}(\alpha_{i-1})^{x_{i}}(\alpha_{0})^{x_{1}}(\alpha_{1})^{x_{2}}...(\alpha_{r-1})^{x_{r}}}{(x_{1}+x_{2}+...+x_{i}+...+x_{r}+\beta)^{rk-n}}.
  \end{eqnarray}}
  Also, using\eqref{eq5} $E(X_{i}^{\ell})$ can be written as
 \begin{equation}\label{eq104} 
   E(X_{i}^{\ell})=\sum\limits_{x_{i}=0}^{\infty} x_{i}^{\ell}(\alpha_{i-1})^{x_{i}}\left(\frac{M^{(r)}_{n}(\beta+x_{i}+1;k;\overline{\alpha}_{r})-M^{(r)}_{n}(\beta+x_{i};k;\overline{\alpha_{r}})}{M^{(r)}_{n}(\beta;k;\overline{\alpha})}\right).
\end{equation}
\\
   Using the method given by Gupta \cite{Gupta}, we obtain  $\ell-$th factorial moments as follow\\
  From \eqref{eq2} \\
   \begin{equation*}
   M^{(r)}_{n}(\beta;k;\overline{\alpha})=\frac{(-1)^{r}2^{r(1-k)}n!}{(n-rk)!}\sum\limits_{x_{1},x_{2},...,x_{r}=0}^{\infty}\frac{(\alpha_{0})^{x_{1}}(\alpha_{1})^{x_{2}}...(\alpha_{r-1})^{x_{r}}}{(x_{1}+x_{2}+...+x_{r}+\beta)^{rk-n}},
  \end{equation*}
  \\
 the $\ell$-th factorial moment can be obtained by partial differential of $\ell$ times of the previous equation with respect to $\alpha_{i-1}$ as following
 {\footnotesize \begin{multline}
  \frac{(\alpha_{i-1})^{\ell}\frac{\partial^{\ell}}{\partial \alpha_{i-1}^{\ell}}M^{(r)}_{n}(\beta;k;\overline{\alpha})}{M^{(r)}_{n}(\beta;k;\overline{\alpha})}=E(X^{[\ell]})\\
 =\frac{(-1)^{r}2^{r(1-k)n!}}{(n-rk)!M_{n}^{(r)}(\beta;k;\overline{\alpha_{r}})}\sum\limits_{x_{1},x_{2},...,x_{i},...,x_{r}=0}^{\infty}\frac{(x_{i})_{\ell}(\alpha_{i-1})^{x_{i}}(\alpha_{0})^{x_{1}}(\alpha_{1})^{x_{2}}...(\alpha_{r-1})^{x_{r}}}{(x_{1}+x_{2}+...+x_{i}+...+x_{r}+\beta)^{rk-n}},
  \end{multline}}
  where $x_{i})_{\ell}=(x_{i}(x_{i}-1)...(x_{i}-\ell+1).$
  Since
  \begin{equation*}
  (x_{i})_{\ell}=\sum\limits_{j=0}^{\ell}s(\ell,j)x_{i}^{j},
  \end{equation*}
  where $s(\ell,j)$ are Stirling numbers of first kind see \cite{Comtet}.\\
  Hence, we can obtain relation between moments and factorial moments as follows
  \begin{equation}\label{eq106}
  E(X_{i}^{[\ell]})=\sum\limits_{j=0}^{\ell}s(\ell,j)E(X_{i}^{\ell}).
  \end{equation}  
\begin{theorem} Setting $\textbf{X}=(X_{1},X_{2},...,X_{r})$ and $\textbf{t}=(t_{1},t_{2},...,t_{r})$ and \\ $\textbf{t}\overline{\alpha}_{r}=(t_{1}\alpha_{0},t_{2}\alpha_{2},...,t_{r}\alpha_{r-1})$. The probability generating function of \textbf{UGAT} distribution is
\begin{equation}\label{eq7}
G_{\textbf{X}}(t_{1},t_{2},...,t_{r})=\frac{M^{(r)}_{n}(\beta;k;\textbf{t}\overline{\alpha}_{r})}{M^{(r)}_{n}(\beta;k;\overline{\alpha}_{r})}.
\end{equation} 
\end{theorem} 
\subsection{\textbf{Bivariate \textbf{UGAT} distribution}}
In the case $r=2$ in \eqref{eq1} for any two random variables $X_{i},X_{j}$, $i \neq j,$ we have
\subsubsection{\textbf{Marginal conditional probability distribution}}
{ \begin{multline}\label{eq8}
P(X_{i}=x_{i}|X_{j}=x_{j})=\frac{P(X_{i}=x_{i},X_{j}=x_{j})}{P(X_{j}=x_{j})}\\
=\frac{2^{2(1-k)}n!(\alpha_{i})^{x_{i}}}{(n-2k)!(x_{i}+x_{j})^{2k-n}\left(M^{(2)}_{n}(\beta+x_{i}+1;k;\overline{\alpha}_{2})-M^{(2)}_{n}(\beta+x_{i};k;\overline{\alpha}_{2})\right)}.
\end{multline}}
\subsubsection{\textbf {Marginal expectation}}
\begin{multline}\label{eq9}
E(X_{i}=x_{i}|X_{j}=x_{j})=\sum\limits_{x_{i}=0}^{\infty}x_{i}P(X_{i}=x_{i}|X_{j}=x_{j})\\
=\frac{2^{2(1-k)}n!}{(n-rk)!}\sum\limits_{x_{i}=0}^{\infty}\frac{x_{i}(\alpha_{i})^{x_{i}}}{(x_{i}+x_{j})^{2k-n}\left(M^{(2)}_{n}(\beta+x_{i}+1;k;\overline{\alpha}_{2})-M^{(2)}_{n}(\beta+x_{i};k;\overline{\alpha}_{2})\right)}.
\end{multline}
\section{Reliability concepts for UGAT distributions}
\subsection{Multivariate reliability function}
\begin{theorem} Let $\textbf{X}$ $=(X_{1},X_{2},...,X_{r})$ be a discrete random vector, $\mathbf{x}$ $=(x_{1},x_{2},...,x_{r})$ $\in \textbf{R}_{+}^{r}$ representing the lifetimes of $r-$component system with the multivariate reliability function 
\begin{equation}\label{eq9}
R(\mathbf{x})=\frac{\prod\limits_{i=1}^{r-1}(\alpha_{i-1})^{x_{i}} M^{(r)}_{n}(\beta+x_{1}+x_{2}+...+x_{r};k;\overline{\alpha}_{r})}{M^{(r)}_{n}(\beta;k;\overline{\alpha}_{r})}
\end{equation}
\end{theorem} 
\begin{proof} From the definition of the multivariate reliability function, see \cite{Asha} and \eqref{eq2}
{\small \begin{eqnarray*}
R(x_{1},x_{2},...,x_{r})&=&P(X_{1} \geq x_{1},X_{2} \geq x_{2},...,X_{r} \geq x_{r})\\
&=&\sum\limits_{m_{1}\geq x_{1}}\sum\limits_{m_{2}\geq x_{2}}...\sum\limits_{m_{r}\geq x_{r}}P(X_{1}=m_{1},X_{2}=m_{2},...,X_{r}=m_{r})\\
&=&\frac{(-1)^{r}2^{r(1-k)}n!}{(n-rk)!M^{(r)}_{n}(\beta;k;\overline{\alpha}_{r})}\sum\limits_{m_{1}\geq x_{1}}^{\infty}\sum\limits_{m_{2}\geq x_{2}}^{\infty}...\sum\limits_{m_{r}\geq x_{r}}^{\infty}\frac{(\alpha_{0})^{m_{1}}(\alpha_{1})^{m_{2}}..(\alpha_{i-1})^{m_{i}}...(\alpha_{r-1})^{m_{r}}}{(m_{1}+m_{2}+...+m_{r}+\beta)^{rk-n}},
\end{eqnarray*}}
substituting $m_{i}-x_{i}=\ell_{i}$, $ i=1,2,...,r$ in the previous equation, we obtain {\eqref{eq9}}.
\end{proof}
\begin{remark} Let $\textbf{X}$ $=(X_{1},X_{2},...,X_{r})$ be a discrete random vector, $\mathbf{x}$ $=(x_{1},x_{2},...,x_{r})$ $\in \textbf{R}_{+}^{r}$. The marginal survival function, see \cite{chu}, is defined as
\begin{eqnarray*}
R_{i}(t_{i},\textbf{X})&=&P(X_{i}-x_{i} >t_{i}|\mathbf{X \geq x})\\
&=&\frac{R(x_{1},x_{2},...,x_{i}+t_{i},...,x_{r})}{R(x_{1},x_{2},...,x_{r})},
\end{eqnarray*}
so we obtain the marginal survival function of UGAT distribution
\begin{equation}\label{eq30}
R_{i}(t_{i},\textbf{X})=(\alpha_{i-1})^{t_{i}}\frac{M_{n}^{(r)}(x_{1}+x_{2}+...+x_{i}+t_{i}+...+x_{r}+\beta;k;\overline{\alpha}_{r})}{M_{n}^{(r)}(x_{1}+x_{2}+...+x_{r}+\beta;k;\overline{\alpha}_{r})}.
\end{equation}
\end{remark}
\begin{theorem} $R(\mathbf{X})$ is said to be\\
\\
i) Multivariate new better than used (Multivariate new worse than used ) MNBU (MNWU) if
\begin{multline}\label{eq0006}
M^{(r)}_{n}(\beta;k;\overline{\alpha}_{r})M^{(r)}_{n}(\mathbf{X_{r}+t}+\beta;k;\overline{\alpha}_{r})\leq(\geq)M^{(r)}_{n}(\mathbf{X_{r}}+\beta;k;\overline{\alpha}_{r})M^{(r)}_{n}(\mathbf{t}+\beta;k;\overline{\alpha}_{r}).
\end{multline}
ii) Multivariate new better than used in expectation (Multivariate new worse than used in expectation ) MNBUE (MNWUE) if
\begin{multline}\label{eq0007}
M^{(r)}_{n}(\beta;k;\overline{\alpha}_{r})\sum\limits_{t_{1},t_{2},...,t_{r}=0}^{\infty}\prod_{i=0}^{r-1}\alpha_{i}M^{(r)}_{n}(\mathbf{X_{r}+t}+\beta;k;\overline{\alpha}_{r})\leq(\geq)\\
M^{(r)}_{n}(\mathbf{X_{r}}+\beta;k;\overline{\alpha}_{r})\sum\limits_{t_{1},t_{2},...,t_{r}=0}^{\infty}\prod_{i=0}^{r-1}\alpha_{i}M^{(r)}_{n}(\mathbf{t}+\beta;k;\overline{\alpha}_{r}),
\end{multline}
where $\textbf{X}$ $=(X_{1},X_{2},...,X_{r})$ be a discrete random vector, $ \mathbf{X_{r}}=x_{1}+x_{2}+...+x_{r},$ 
$ \mathbf{t}=t_{1}+t_{2}+...+t_{r}.$
\end{theorem}
\begin{proof} When 
\begin{multline*}
M^{(r)}_{n}(\mathbf{X_{r}+t}+\beta;k;\overline{\alpha}_{r})\leq\frac{M^{(r)}_{n}(\mathbf{X_{r}}+\beta;k;\overline{\alpha}_{r})M^{(r)}_{n}(\mathbf{t}+\beta;k;\overline{\alpha}_{r})}{M^{(r)}_{n}(\beta;k;\overline{\alpha}_{r})}\\
\frac{M^{(r)}_{n}(\mathbf{X_{r}+t}+\beta;k;\overline{\alpha}_{r})}{M^{(r)}_{n}(\beta;k;\overline{\alpha}_{r})}\leq\frac{M^{(r)}_{n}(\mathbf{X_{r}}+\beta;k;\overline{\alpha}_{r})}{M^{(r)}_{n}(\beta;k;\overline{\alpha}_{r})} \frac{M^{(r)}_{n}(\mathbf{t}+\beta;k;\overline{\alpha}_{r})}{M^{(r)}_{n}(\beta;k;\overline{\alpha}_{r})}\\
\prod_{i=1}^{r}(\alpha_{i-1})^{x_{i}+t_{i}}\frac{M^{(r)}_{n}(\mathbf{X_{r}+t}+\beta;k;\overline{\alpha}_{r})}{M^{(r)}_{n}(\beta;k;\overline{\alpha}_{r})}\leq\\
\prod_{i=1}^{r}(\alpha_{i-1})^{x_{i}}\frac{M^{(r)}_{n}(\mathbf{X_{r}}+\beta;k;\overline{\alpha}_{r})}{M^{(r)}_{n}(\beta;k;\overline{\alpha}_{r})}\prod_{i=1}^{r}(\alpha_{i-1})^{t_{i}}\frac{M^{(r)}_{n}(\mathbf{t}+\beta;k;\overline{\alpha}_{r})}{M^{(r)}_{n}(\beta;k;\overline{\alpha}_{r})},
\end{multline*}
hence, we get\\
$$
R(x_{1}+t_{1},x_{2}+t_{2},...,x_{r}+t_{r})\leq R(x_{1},x_{2},...,x_{r})R({t_{1},t_{2},...,t_{r}}).
$$
From the definition of the Multivariate new better used, see \cite{Hana}, then $R(\mathbf{X})$ is MNBU.\\
Similarly, from the definition of MNBUE (multivariate new better used in expectation),see \cite{Hana},\\
$$
\frac{\sum\limits_{t_{1},t_{2},...,t_{r}=0}^{\infty}R(x_{1}+t_{1},x_{2}+t_{2},...,x_{r}+t_{r})}{R(x_{1},x_{2},...,x_{r})}\leq \sum\limits_{t_{1},t_{2},...,t_{r}=0}^{\infty}R(t_{1},t_{2},...,t_{r}),
$$\\
we obtain \eqref{eq0007}.

\end{proof}
\subsection{Multivariate hazard rate function} 
We consider multivariate hazard rate function\cite{Asha}\\
$$
h(\textbf{x})=\left(h_{1}(\textbf{x}),h_{2}(\textbf{x}),...,h_{r}(\textbf{x})\right)
$$
\begin{eqnarray*}
h_{i}(\mathbf{x})&=&P(X_{i}=x_{i}|\textbf{X} \geq \textbf{x})\\
&=&1-\frac{R(x_{1},x_{2},...,x_{i-1},x_{i}+1,x_{i+1},...,x_{r})}{R(x_{1},x_{2},...,x_{r})},
\end{eqnarray*}
where $\mathbf{x}$ $=(x_{1},x_{2},...,x_{r})$ $\in \textbf{R}_{+}^{r}$, so we obtain the following theorem.
\begin{theorem} The multivariate hazard rate function of UGAT distribution is given by
\begin{equation}\label{eq10}
h_{i}(\mathbf{x})=1-\alpha_{i}\frac{M_{n}^{(r)}(x_{1}+x_{2}+...+x_{i-1}+x_{i}+1+x_{i+1}+...+x_{r}+\beta;k;\overline{\alpha}_{r})}{M_{n}^{(r)}(x_{1}+x_{2}+...+x_{r}+\beta;k;\overline{\alpha}_{r})}
\end{equation}
\end{theorem}
\begin{theorem} Let\\
 $\mathbf{X^{'}_{r}}=(x_{1}+x_{2}+...+x_{i-1}+x_{i}+1+x_{i+1}+...+x_{r}),$ $\mathbf{X_{r}}=(x_{1}+x_{2}+...+x_{r})$ and $\mathbf{t}=t_{1}+t_{2}+...+t_{r},$\\
\\
then the following statement about UGAT distribution holds\\
\\
 UGAT distribution is $\mathbf{MIFR}$ $(\mathbf{MDFR}) $ iff 
\begin{equation}\label{eq00005}
\frac{M^{(r)}_{n}(\beta+\mathbf{X^{'}_{r}+t};k;\overline{\alpha}_{r})}{M^{(r)}_{n}(\beta+\mathbf{X_{r}+t};k;\overline{\alpha}_{r})} \leq (\geq) \frac{M^{(r)}_{n}(\beta+\mathbf{X^{'}_{r}};k;\overline{\alpha}_{r})}{M^{(r)}_{n}(\beta+\mathbf{X_{r}};k;\overline{\alpha}_{r})}.
\end{equation}
\end{theorem}
\begin{proof} when \\

$$
\frac{M^{(r)}_{n}(\beta+\mathbf{X^{'}_{r}+t};k;\overline{\alpha}_{r})}{M^{(r)}_{n}(\beta+\mathbf{X_{r}+t};k;\overline{\alpha}_{r})} \leq \frac{M^{(r)}_{n}(\beta+\mathbf{X^{'}_{r}};k;\overline{\alpha}_{r})}{M^{(r)}_{n}(\beta+\mathbf{X_{r}};k;\overline{\alpha}_{r})},
$$
\\
then \\
$$
1-\alpha_{i}\frac{M^{(r)}_{n}(\beta+\mathbf{X^{'}_{r}+t};k;\overline{\alpha}_{r})}{M^{(r)}_{n}(\beta+\mathbf{X_{r}+t};k;\overline{\alpha}_{r})} \geq 1-\alpha_{i}\frac{M^{(r)}_{n}(\beta+\mathbf{X^{'}_{r}};k;\overline{\alpha}_{r})}{M^{(r)}_{n}(\beta+\mathbf{X_{r}};k;\overline{\alpha}_{r})},
$$\\
so
$$
h(x_{1}+t_{1},x_{2}+t_{2},...,x_{r}+t_{r}) \geq h(x_{1},x_{2},...,x_{r}).
$$
\\
\\
Hence UGAT distribution is MIFR (multivariate increasing failure rate), see \cite{Asha}.
\end{proof}
\subsection{ Multivariate mean residual life}
 Let $\textbf{X}$ $=(X_{1},X_{2},...,X_{r})$ be a discrete random vectors, $\textbf{x}$ $=(x_{1},x_{2},...,x_{r})$ $\in \textbf{R}_{+}^{r}$. We consider \textbf{MMRL}$(m)$ \cite{Asadi}\\
$$
m(\textbf{x})=\left(m_{1}(\textbf{x}),m_{2}(\textbf{x}),...,m_{r}(\textbf{x})\right)
$$
\begin{eqnarray*}
m_{i}(\textbf{x})&=&E(X_{i}-x_{i}|\textbf{X} \geq \textbf{x})\\
&=&\sum\limits_{t_{i}=0}^{\infty}R_{i}(t_{i},\textbf{X}),
\end{eqnarray*} 
so we get the following theorem
\begin{theorem} \textbf{MMRL}$m$ of UGAT distribution is given by
\begin{equation}\label{eq11}
 MMRL(x_{i})=\sum\limits_{t_{i}=0}^{\infty}\frac{(\alpha_{i-1})^{t_{i}}M^{(r)}_{n}(\beta+x_{1}+x_{2}+...+x_{i}+t_{i}+...+x_{r};k;\overline{\alpha}_{r})}{M^{(r)}_{n}(\beta+x_{1}+x_{2}+...+x_{r};k;\overline{\alpha}_{r})}.
\end{equation}
\end{theorem}

\section{\textbf{Maximum likelihood estimators}}
We use the maximum likelihood method to estimate the unknown parameters of UGAT distribution. Consider constant values to $r,n$, we want to estimate the  parameter $\alpha_{i}$, $i=1,2,...,r$, $\beta$ and $k$. Suppose we have a sample of size $N$ in the form $((x_{11},x_{21},...,x_{r1}),(x_{12},x_{22},...,x_{r2}),...,$\\
$(x_{1N},x_{2N},...,x_{rN}))$ from UGAT distribution. Based on observations, the Likelihood function of the sample given by
\begin{equation}\label{eq11}
L(\alpha_{i},\beta,k)=\prod\limits_{i=1}^{n}C\frac{(\alpha_{0}^{x_{1i}})(\alpha_{1}^{x_{2i}})...(\alpha_{r-1}^{x_{ri}})}{(x_{1i}+x_{2i}+...+x_{ri})^{rk-n}M^{(r)}_{n}(\beta;k;\overline{\alpha}_{r})},
\end{equation}
where $C=\prod\limits_{i=0}^{N}\frac{(-1)^{r}2^{r(1-k)}n!}{(n-rk)!}$.\\
Log-Likelihood function can be written as
{\footnotesize \begin{equation}\label{eq12}
\log L=\log C+\sum\limits_{j=1}^{r}\sum\limits_{i=1}^{N}x_{ji}\log(\alpha_{j-1})+(n-rk)\sum\limits_{i=1}^{N}\log(x_{1i}+x_{2i}+...+x_{ri}+\beta)-N\log(M^{(r)}_{n}(\beta;k;\overline{\alpha}_{r})).
\end{equation}}
Computing the first partial derivatives of \eqref{eq12} with respect to $\alpha_{i},$ $i=0,1,...,r-1$, $\beta$ and $k$, we get the normal equations
{ \small \begin{eqnarray*}
\frac{\partial}{\partial \alpha_{i-1}}\log L&=&\frac{\sum\limits_{j=1}^{N}x_{ij}}{\alpha_{i-1}}-\frac{N}{\alpha_{i-1}}E(X_{i}),  i=1,2,...,r-1  \\
\frac{\partial}{\partial \beta}\log L&=&-(rk-n)\sum\limits_{1}^{N}\frac{1}{x_{1i})+x_{2i}+...+x_{ri}}+N(rk-n)E((X_{1}+X_{2}+...+X_{r}+\beta)^{-1})   \\
\frac{\partial}{\partial k} \log L&=&-r \sum\limits_{i=1}^{N} \log(x_{1i}+x_{2i}+...+x_{ri}+\beta)+Nr E(\log(X_{1})+X_{2}+...+X_{r}+\beta)) 
\end{eqnarray*}}
In the case of $r=3$, we study the bivariate UGAT distribution to get the MLEs of the parameters $\alpha_{0},\alpha_{1},\alpha_{2}$ and $\beta$. For that we have to solve the above system of four non-linear equations with respect to $\alpha_{0},\alpha_{1},\alpha_{2}$ and $\beta$. The normal equations are
{\small \begin{eqnarray}
\frac{\partial}{\partial \alpha_{0}}\log L&=&\frac{\sum\limits_{j=1}^{N}x_{ij}}{\alpha_{0}}-\frac{N}{\alpha_{0}}E(X_{1}),   \\
\frac{\partial}{\partial \alpha_{1}}\log L&=&\frac{\sum\limits_{j=1}^{N}x_{ij}}{\alpha_{1}}-\frac{N}{\alpha_{1}}E(X_{2}),   \\
\frac{\partial}{\partial \alpha_{2}}\log L&=&\frac{\sum\limits_{j=1}^{N}x_{ij}}{\alpha_{2}}-\frac{N}{\alpha_{2}}E(X_{X_{3}}),   \\
\frac{\partial}{\partial \beta}\log L&=&-(2k-n)\sum\limits_{1}^{N}\frac{1}{x_{1i}+x_{2i}+x_{3i}+\beta}+N(2k-n)E((X_{1}+X_{2}+\beta)^{-1})   
\end{eqnarray}}
These equations are not easy to solve, so numerical technique is needed to get the MLEs. The approximate confidence intervals of the parameters based on asymptotic distributions of their MLEs are derived.
\section{Data analysis}
In this section, we present the analysis of a real data using the UGAT model and compare it with model like multivariate poisson-log normal $P\bigwedge^3$ model. The following data represent the study of the relative effectiveness of three different air samplers $1,2,3$ to detect pathogenic bacteria triplets of a microbiologist obtained triplets of bacterial colony counts $x_{1},x_{2},x_{3}$ from samplers $1,2,3$ d the data in each of $50$ different sterile locations. This data is available in \cite{Ho} and the data are represented in the following \textbf{Table 1}.
\begin{table}[h!]
\caption{Bacterial counts by 3 samplers in 50 sterile locations}
\centering
\begin{tabular}{ | c  c  c | c  c  c | c  c  c | c c c | c c c | }
\hline\hline
$X_{1}$ & $X_{2}$ & $ X_{3}$ & $X_{1}$& $X_{2}$ & $ X_{3}$& $X_{1}$& $X_{2}$ & $ X_{3}$ & $X_{1}$ & $X_{2}$ & $ X_{3}$ &$X_{1}$ & $X_{2}$ & $ X_{3}$ \\ 
\hline
1 & 2 & 11 & 3 & 6 & 6 & 3 & 8 & 2 & 7 & 10 & 5 & 22 & 9 & 9 \\
8 & 6 & 0 & 3 & 9 & 14 & 1 & 1 & 30 & 2 & 2 & 8 & 5 & 2 & 4\\
2 & 13 & 5 & 4 & 2 & 25& 4 & 5 & 15 & 3 & 15 & 3 & 2 & 0 & 6 \\
2 & 8 & 1 & 9 & 7 & 3 & 7 & 6 & 3 & 1 & 8 & 2 & 2 & 1 & 1 \\
5 & 6 & 5 & 5 & 4 & 8 & 8 & 10 & 4 & 4 & 6 & 0 & 4 & 6 & 4 \\
14 & 1 & 7 & 4 & 4 & 7 & 3 & 2 & 10 & 8 & 7 & 3 & 4 & 9 & 2 \\
3 & 9 & 2 & 7 & 3 & 2 & 6 & 8 & 5 &6 & 6 & 6 & 8 & 4 & 6 \\
7 & 6 & 8 & 1 & 14 & 6 & 2 & 3 & 10 & 4 & 14 & 7 & 3 & 10 & 6 \\
3 & 4 & 12 & 2 & 13 & 0 & 1 & 7 & 3 & 3 & 3 & 14 & 4 & 7 & 10 \\
1 & 9 & 7 & 14 & 9 & 5 & 2 & 9 & 12 & 6 & 8 & 3 & 2 & 4 & 6 \\
\hline
\end{tabular}
\end{table}

 The UGAT model is used to fit this data set. The MLE(s) of unknown parameter(s) are $\alpha_{0}=.787, \alpha_{1}= .846, \alpha_{2}= .849, \beta=954.707 $.\\
 The value of log-likelihood, Akaike information criterion(AIC) and Bayesian information criterion (BIC) test statistic two different models are given in \textbf{Table 2}
 \\
 \begin{table}[h!]
\caption{The MLE(s) parameters, Log-likelihood, AIC and BIC}
\centering
\begin{tabular}{ | c | c | c | c | c | }
\hline\hline
\textbf{The model or hypothesis} & \textbf{parameters} & \textbf{-L} & \textbf{AIC} & \textbf{BIC} \\
\hline
P$ \bigwedge^{3}$ model & 9 &397.8 & 813.6 & 810.89\\
P$ \bigwedge^{3}$ with equal $\mu$  & 7 & -402 & 818 & 815.89\\
Equicovariance & 5 & 401.1 & 812.2 & 810.69 \\
Isotropic & 3 & 404 & 814 & 813.097 \\
UGAT model &4 & 401.797& 811.594 & 810.39\\
\hline
\end{tabular}
\end{table}
\\ 
\textbf{Table 2} shows that The UAGT model fits certain well-known data sets better than  P$ \bigwedge^{3}$ model. Reducing the number of parameters to four by fixing one of parameters, therefore $k$ still provides a better fit than existing model.
\bibliographystyle{elsarticle-num}

\end{document}